\documentclass{amsart}

\usepackage[english]{babel}
\usepackage[latin1]{inputenc}
\usepackage[colorlinks]{hyperref}

\usepackage{pgfplots}
\pgfplotsset{compat=1.10}
\usepgfplotslibrary{fillbetween}
\usetikzlibrary{patterns}

\usepackage{todonotes}

\allowdisplaybreaks

\usepackage{amssymb,amsmath,amsthm,amsfonts, mathrsfs, mathtools}
\usepackage{graphicx}
\usepackage{enumitem}

\def\N{\mathbb{N}}
\def\R{\mathbb{R}}

\def\loc{\mathrm{loc}}

\usepackage{xcolor}
\usepackage[colorlinks]{hyperref}
\hypersetup{linkcolor=black,citecolor=black,filecolor=black,urlcolor=blue}

\usepackage{fullpage}
\usepackage{tikzsymbols}

\DeclareMathOperator{\dist}{dist}

\newcommand{\pa}{\partial}

\newtheorem{proposition}{Proposition}[section]
\newtheorem{theorem}[proposition]{Theorem}
\newtheorem{corollary}[proposition]{Corollary}
\newtheorem{lemma}[proposition]{Lemma}
\newtheorem{question}[proposition]{Question}

\theoremstyle{definition}

\newtheorem{remark}[proposition]{Remark}
\numberwithin{equation}{section}

\renewcommand{\le}{\leqslant}

\renewcommand{\ge}{\geqslant}

\newcommand{\e}{\varepsilon}
\renewcommand{\epsilon}{\varepsilon}

\newcommand{\supp}{\mathrm{supp}}

\makeatletter
\@namedef{subjclassname@2020}{%
  \textup{2020} Mathematics Subject Classification}
\makeatother

\title{A fractional Hopf Lemma for sign-changing solutions}

\author[S.~Dipierro]{Serena Dipierro}  \thanks{}
\address{Serena Dipierro \newline \indent 
Department of Mathematics and Statistics, University of Western Australia, \newline \indent
35 Stirling Highway, Crawley, Perth}
\email{serena.dipierro@uwa.edu.au}

\author[N.~Soave]{Nicola Soave} 
\address{Nicola Soave \newline \indent
	Dipartimento di Matematica ``Giuseppe Peano",  Universit\`a degli Studi di Torino,  \newline \indent
	Via Carlo Alberto 10, 10123 Torino, Italy}
\email{nicola.soave@unito.it}

\author[E.~Valdinoci]{Enrico Valdinoci}  \thanks{}
\address{Enrico Valdinoci \newline \indent 
Department of Mathematics and Statistics, University of Western Australia, \newline \indent
35 Stirling Highway, Crawley, Perth}
\email{enrico.valdinoci@uwa.edu.au}

\thanks{Nicola~Soave is partially supported by the INdAM - GNAMPA Project, cod. CUP\_\,E53C22001930001 ``Regolarit\`a e singolarit\`a in problemi con frontiere libere". Enrico Valdinoci is supported by the Australian Laureate Fellowship FL190100081
``Minimal surfaces, free boundaries and partial differential equations''. It is a pleasure to thank Ovidiu Savin for inspiring discussions.}

\begin{document}

\begin{abstract}
In this paper we prove some results on the boundary behavior of solutions to fractional elliptic problems. Firstly, we establish a Hopf Lemma for solutions to some integro-differential equations. The main novelty of our result is that we do not assume any global condition on the sign of the solutions. Secondly, we show that non-trivial radial solutions cannot have infinitely many zeros accumulating at the boundary.

We provide concrete examples to show that the results obtained are sharp.
\end{abstract}

\date{\today}
\subjclass[2020]{}

\maketitle

\section{Introduction}

In this paper we prove the validity of a Hopf Lemma for sign-changing solutions of some nonlocal elliptic problems under an additional growth assumption that cannot be removed. 

The motivation for our study comes from the following natural problem:

\begin{question} \label{LADOMA}
Suppose that we have a sign-changing solution of a fractional linear elliptic problem in a domain. Suppose moreover that~$u$ does not change sign in a neighborhood of a boundary point~$x_0$, with~$u(x_0)=0$. Is it true that the fractional normal derivative of~$u$ at the boundary cannot vanish,
unless~$u \equiv 0$? \end{question}

The answer to Question~\ref{LADOMA}
is affirmative for the corresponding local issue, where the Hopf Lemma holds under merely local assumptions (see e.g. \cite[Lemma~3.4]{GiTr}). In contrast, as far as we know, the nonlocal versions of the Hopf Lemma available in the literature so far require either
the positivity of $u$ in the whole space (or some sign-condition for potential of source terms; see e.g. \cite{GrSe} and \cite[Lemma~7.3]{RosOtsur}), or the antisymmetry of $u$ across an hyperplane (see e.g. \cite[Proposition~3.3]{FaJa},
\cite[Proposition~2.2]{SoVa} and~\cite[Lemma 4.1]{ciraolo2021symmetry}). This prevents any direct application to sign changing solutions in domains, which is the main reason for which Question~\ref{LADOMA} is still open. 

In our first main result, namely Theorem~\ref{thm: hopf}, we will give an affirmative answer to Question~\ref{LADOMA} under an additional ``second order fractional'' growth assumption. Then, in Theorem~\ref{thm: hopf:NO}, we will show that
this additional assumption cannot be removed.


We stress that
even believing that the answer to Question~\ref{LADOMA} were positive under a suitable additional local assumption was not obvious. Indeed, the maximum or minimum principle for integro-differential equations holds only under global assumptions, and, very differently from the classical case,
solutions of fractional elliptic equations are known to have local extrema (see \cite{MR3626547}). Furthermore, the classical Harnack inequality fails in the nonlocal setting, unless the positivity of the function is assumed in the whole space (see e.g.~\cite{MR1941020}, \cite{MR2817382} and~\cite[Theorem 3.3.1]{MR3469920}). Therefore,
the analysis of qualitative properties of solutions to nonlocal problems
which relies on the maximum principle cannot be separated from a global information on 
the sign of the solution itself and therefore these types of arguments typically do not go through for
sign-changing solutions (see e.g. \cite{CoIa}). 

Since the maximum principle and the Harnack inequality are conceptually strictly linked to the Hopf Lemma, the failure of these tools for sign-changing solutions in nonlocal cases
might suggest that also Question~\ref{LADOMA} had a negative answer (and this even in the presence of additional assumptions of local nature):
hence we believe that the positive answer that we instead provide in
Theorem~\ref{thm: hopf} is genuinely interesting.
We also think that the fact that the additional growth assumption
cannot be removed, as established in Theorem~\ref{thm: hopf:NO},
provides a nice clarification in terms of the importance of this type of hypotheses.
\medskip

To state our first main result in a precise form, we consider the class of integro-differential operators $L$ defined\footnote{Here we skate around the minor regularity requirements on~$w$ in order
to write~\eqref{def L} pointwise: at this level, we are implicitly assuming~$w$ to
be ``regular enough'', but a more precise setting will be discussed in the forthcoming
Remark~\ref{REGOLL}.} by
\begin{equation}\label{def L}
L w(x) = PV \int_{\R^n} \frac{w(x)-w(x+y)}{|y|^{n+2s}} a\left(\frac{y}{|y|}\right) dy
,\end{equation}
where $PV$ stands for ``principal value", $s\in(0,1)$ and~$a:S^{n-1} \to [0,+\infty)$ is a
nonnegative function in~$L^\infty(S^{n-1})$ such that 
\begin{equation}\label{dieur4bvi}
a(\theta) = a (-\theta) \end{equation}
for every $\theta \in S^{n-1}$ (the unit sphere in $\R^n$), and, for some positive constants $0<\lambda \le \Lambda$,
\begin{equation}\label{def a}
\lambda \le a(\theta) \le \Lambda \qquad \text{for all~$\theta\in S^{n-1}$}.
\end{equation}
Integro-differential operators of the form~\eqref{def L} arise naturally in the study of stochastic processes with jumps, and have been widely studied both in Probability and in Analysis and PDEs. We refer the reader to the introduction of \cite{RosOtsur} and to the references therein for more details.
\medskip

Let us now introduce a natural functional space to look at when dealing with operators as in~\eqref{def L}. 
Given~$\Omega\subseteq\R^n$, we denote by~$ {\mathcal{Z}}_\Omega$ the space of continuous functions~$u:\R^n\to\R$ such that
$$ \int_{\R^n} \frac{|u(x)|}{1+|x|^{n+2s}}\,dx <+\infty$$
and for which the following limit exists for every~$x\in\Omega$:
$$ \lim_{\varepsilon\searrow0} \int_{\R^n\setminus B_\varepsilon} \frac{u(x)-u(x+y)}{|y|^{n+2s}} a\left(\frac{y}{|y|}\right) dy.$$

\medskip

We now describe in further detail the type of additional assumption that ensures the validity of Question~\ref{LADOMA},
stating it in terms of the interior sphere condition. Namely, as customary,
given an open subset~$\Omega$ of~$\R^n$ and
a point~$x_0\in\partial\Omega$, we say that~$\Omega$
satisfies an interior sphere condition at $x_0$
if there exists $\bar r >0$ such that, for every~$r \in (0,\bar r]$, there exists a
ball $B_r(x_r)\subseteq\Omega$ with $x_0\in(\pa \Omega)\cap(\partial B_r(x_r))$.

In this setting,
given a continuous function~$u:\R^n\to\R$ such that~$u(x_0)=0$ and $|u|>0$ in $B_r(x_r)$ for every sufficiently small $r$, we say that \emph{$u$ grows
faster than the power~$2s$ at~$x_0$} if
\begin{equation}\label{OLOms}
\limsup_{r\to0^+} \Phi(r)=+\infty,\qquad{\mbox{where }}\;\Phi(r):=\frac{\displaystyle \inf_{B_{r/2}(x_r)} |u|}{r^{2s}}.
\end{equation}

This is the additional growth condition needed for an affirmative answer to Question~\ref{LADOMA},
according to the following result:

\begin{theorem}\label{thm: hopf}
Let $\Omega$ be an open subset of~$\R^n$, and let~$x_0\in\partial\Omega$.
Assume that~$\pa \Omega$ satisfies an interior sphere condition at $x_0$.

Let $L$ be any operator of the form~\eqref{def L}-\eqref{dieur4bvi}-\eqref{def a}, and let $u\in{\mathcal{Z}}_\Omega$
be such that~$u(x_0) = 0$ and $u^- \in L^\infty(\R^n)$.

Suppose that 
\[
L u \ge V(x) u \quad \text{in $\Omega$}
\]
pointwise, where~$V \in L^{1}_{\loc}(\Omega)$ is such that $V^- \in L^\infty(\Omega)$.

Suppose moreover that there exists~$R>0$ such that $u \ge 0 $ in $B_R(x_0)$, $u>0$ in $B_R(x_0) \cap \Omega$, and that $u$ grows
faster than the power~$2s$ at~$x_0$.

Then, for every~$\beta\in(0,\pi/2)$ we have
\begin{equation}\label{HOP-0ok}
\liminf_{\Omega\ni x \to x_0} \frac{u(x)}{|x-x_0|^s} >0,
\end{equation}
whenever the angle between $x-x_0$ and the vector joining~$x_0$ and the center of the interior sphere is smaller than $\pi/2-\beta$. 

In particular, if $d(x):= \dist(x, \pa \Omega)$, and $u/ d^s$ can be extended as a continuous function on a neighborhood of $x_0$ in $\pa \Omega$, then
\[
\frac{u}{d^s}(x_0) >0.
\]
\end{theorem}

In Theorem~\ref{thm: hopf} and in the rest of this paper, we use the standard notation $f^+$, $f^-$ for the positive and the negative part of a function~$f$, namely
\begin{eqnarray*}
f^+(x):=\max\{f(x),0\}\qquad{\mbox{and}}\qquad f^-(x):=f^+(x)-f(x)=\max\{ -f(x),0\}.
\end{eqnarray*}

Changing~$u$ with~$-u$ in Theorem~\ref{thm: hopf}, we obtain the following statement:

\begin{corollary}\label{thm: hopf33}
Let $\Omega$ be an open subset of~$\R^n$, and let~$x_0\in\partial\Omega$.
Assume that~$\pa \Omega$ satisfies an interior sphere condition at $x_0$.

Let $L$ be any operator of the form~\eqref{def L}-\eqref{dieur4bvi}-\eqref{def a}, and let $u\in{\mathcal{Z}}_\Omega$
be such that~$u(x_0) = 0$ and $u^+ \in L^\infty(\R^n)$.

Suppose that 
\[
L u \le V(x) u \quad \text{in $\Omega$}
\]
pointwise, where~$V \in L^{1}_{\loc}(\Omega)$ is such that $V^- \in L^\infty(\Omega)$.

Suppose moreover that there exists~$R>0$ such that $u \le 0 $ in $B_R(x_0)$, $u<0$ in $B_R(x_0) \cap \Omega$, and that $u$ grows
faster than the power~$2s$ at~$x_0$.

Then, for every~$\beta\in(0,\pi/2)$ we have
\begin{equation*}
\limsup_{\Omega\ni x \to x_0} \frac{u(x)}{|x-x_0|^s} <0,
\end{equation*}
whenever the angle between $x-x_0$ and the vector joining~$x_0$ and the center of the interior sphere is smaller than $\pi/2-\beta$. 

In particular, if $d(x):= \dist(x, \pa \Omega)$, and $u/ d^s$ can be extended as a continuous function on a neighborhood of $x_0$ in $\pa \Omega$, then
\[
\frac{u}{d^s}(x_0) <0.
\]
\end{corollary}

As a direct consequence of Theorem~\ref{thm: hopf} and Corollary~\ref{thm: hopf33}, we have the following result:

\begin{corollary}\label{CONTRO-0}
Let $\Omega$ be an open subset of~$\R^n$ of class $C^{1,1}$, and let~$x_0\in\partial\Omega$. Let $u \in C^s(\Omega) \cap L^\infty(\Omega)$ be such that 
\[
\begin{cases}
(-\Delta)^s u=V(x) u & \text{in $\Omega$}, \\
u=0 & \text{in $\R^n \setminus \Omega$},
\end{cases} \quad \text{with} \quad u>0 \quad \text{(resp. $u<0$)} \quad \text{in $B_R(x_0) \cap \Omega$}, 
\]
for some $R>0$ and $V \in L^\infty(\Omega)$. 

If $u$ grows
faster than the power~$2s$ at~$x_0$, then 
\[
\frac{|u|}{d^s}(x_0) > 0.
\]
\end{corollary}

Let us now put forth some observations related to the above results.

\begin{remark} We stress that the condition~$V\in L^\infty(\Omega)$ in Corollary~\ref{CONTRO-0} cannot be removed, not even in the case of radial and positive solutions.
See Appendix~\ref{CONTRO} for an explicit example.\end{remark}

\begin{remark} 
For positive solutions of linear problems with $V \in L^\infty(\Omega)$, one always has $u/d^s >0$ at boundary points, by \cite[Lemma 1.2]{GrSe}. Therefore Corollary~\ref{CONTRO-0} is of particular relevance in case of sign-changing solutions.

Note indeed that in our setting we do not assume $u$ to be nonnegative in the whole space $\R^n$, but only to have a sign in the vicinity of a boundary point.\end{remark}

\begin{remark}
We also point out that the assumption that $u^- \in L^\infty(\R^n)$ can be considerably weakened. It is sufficient to suppose that
\begin{equation}\label{cond u-}
u^-(y) \le \bar C\big(1+|y-x_0|^{2s-\delta}\big) \quad {\mbox{ for all }} y \in \R^n,
\end{equation}
for some~$\bar C>$ and~$\delta\in(0,2s)$. For the sake of completeness, we will provide the proof of Theorem \ref{thm: hopf} under the more general assumption~\eqref{cond u-}, rather than~$u^- \in L^\infty(\R^n)$.
\end{remark}

\begin{remark}\label{PRU32S}
Concerning the role of~$\beta$ in Theorem~\ref{thm: hopf},
as customary, a Hopf Lemma needs to avoid ``tangential derivatives'':
in this sense, when~$\Omega$ is a domain of class~$C^1$, the requirement of
Theorem~\ref{thm: hopf} is that the limit in~\eqref{HOP-0ok} occurs along
a given cone of directions which are separated from the tangent hyperplane
(say, by an angle~$\beta$): in particular, in this situation, the limit in~\eqref{HOP-0ok}
can always be taken along the inner normal direction.
\end{remark}

\begin{remark} Comparing~\eqref{OLOms} and~\eqref{HOP-0ok}, we see that Theorem~\ref{thm: hopf} can be considered
as a ``growth improving'' result: very roughly speaking, under a growth assumption with exponent ``slightly better than~$2s$''
as in~\eqref{OLOms}, one obtains a suitable growth assumption in~\eqref{HOP-0ok}
with exponent~$s$ (which is almost ``twice as better'' as the initial assumption)
in a conical subdomain (as detailed in Remark~\ref{PRU32S}).
\end{remark}

\begin{remark}\label{REGOLL}
The regularity assumptions in Theorem~\ref{thm: hopf} (and in particular~\eqref{cond u-}) are rather natural. For instance, let $u$ be a bounded weak solution of a Dirichlet problem of the type
\begin{equation}\label{lin pb st}
\begin{dcases}
Lu = f(x,u) & \text{in $\Omega$}, \\
u = 0 & \text{in $\R^n \setminus \Omega$},
\end{dcases}
\end{equation}
with $\Omega$ open and bounded. If $f \in C^{0,1}_{\loc}(\overline{\Omega} \times \R)$, then $u \in C^{2s+\alpha}(\Omega) \cap C^s(\R^n)$ for some $\alpha>0$, as explained in \cite{RosOtsur}. Thus, $u$ solves the problem
in~\eqref{lin pb st}
pointwise, and, plainly, condition~\eqref{cond u-} holds.

Furthermore, if $\Omega$ is of class $C^{1,1}$, the function~$u/d^s$ can be extended up to~$\pa \Omega$ as a~$C^{0,\alpha}$ function.
More generally, one could also consider distributional solutions, see \cite{RSSejde, RSSeVapoho}.

Finally, we recall that the quantity $u/d^s$ plays the role that the inner normal derivative plays in second-order PDEs. This fact was already observed in the Serrin problem for the fractional Laplacian \cite{DaGV, DiSoVa, FaJa, SoVa}.
\end{remark}

We now observe that the condition that~$u$ grows
faster than the power~$2s$ at~$x_0$ cannot be removed from Theorem~\ref{thm: hopf}:

\begin{theorem}\label{thm: hopf:NO}
Let~$x_0\in\partial B_1$ and~$k\in2\N\cap[2,+\infty)$.
There exist~$\varrho\in(0,1)$ and a compactly supported function~$u\in C(\R^n)\cap C^\infty(B_1)$
such that~$u(x_0) = 0$ with 
\begin{eqnarray*}
&& (-\Delta)^s u=0 \quad \text{in $B_1$,}\\&&
u>0 \quad \text{in $B_{\varrho}(x_0) \setminus\{x_0\}$}\\ {\mbox{and }}
&&\limsup_{x\to x_0}\frac{u(x)}{|x-x_0|^k}=1.
\end{eqnarray*} In particular, the claim in~\eqref{HOP-0ok} does not hold for this~$u$.
\end{theorem}

The previous construction exploits the fact that every functions is locally $s$-harmonic up to a small error, see~\cite{MR3626547}, and requires the possibility of playing with the values of $u$ outside of $B_1$, the set where the equation of $u$ is given. At this point it is natural to wonder what happens if instead we fix such values, for instance posing $u \equiv 0$ in $\R^n \setminus \Omega$. In other words, one may consider the following variant of Question \ref{LADOMA}.

\begin{question} \label{LADOMA2}
Suppose that we have a sign-changing solution of a fractional linear elliptic problem (that is, $(-\Delta)^s u = V u$, for some regular potential~$V$) in a domain, with prescribed exterior datum $u \equiv 0$. Suppose moreover that $u$ does not change sign in a neighborhood $B_R(x_0)$ of a boundary point $x_0$. Is it true that the outer fractional normal derivative of $u$ cannot vanish, unless $u \equiv 0$?
\end{question}

We will give an answer, in the negative, to Question~\ref{LADOMA2}
in the forthcoming Theorem~\ref{INU10}. Before that, let us stress that,
by Theorem \ref{thm: hopf}, the answer to Question~\ref{LADOMA2} is affirmative if $u$ has a strict sign in $\Omega \cap B_R(x_0)$, and $u$ grows faster than the power $2s$ at $x_0$. In trying to better understand the admissible behaviors for $u$, we focus at first on the radial problem in a ball for the fractional Laplacian (in fact, we could consider any radially symmetric subset of $\R^n$, but we focus on the ball for the sake of simplicity). In this setting, we can show the following result.

\begin{theorem}\label{thm: not inf many zeros}
Let $\rho>0$, $s \in (0,1)$, $V \in W^{1,p}(B_\rho)$ for some $p>n/(2s)$, and $x_0 \in \partial B_\rho$.
Let $u \in C^s(\R^n) \cap L^\infty(\R^n)$ be a radial solution to
\begin{equation}\label{radial pb}
\begin{cases}
(-\Delta)^s u = V(x) u & \text{in $B_\rho$}, \\
u = 0 & \text{in $\R^n \setminus B_\rho$}.
\end{cases}
\end{equation}
Suppose that $u$ has infinitely many interior zeros accumulating at the boundary of $B_\rho$; namely, suppose that there exists a
sequence~$\{\rho_m\} \subset \R^+$ such that $\rho_m \to \rho^-$ and $u|_{\pa B_{\rho_m}} = 0$ for every $m$. 

Then $u \equiv 0$ in $\R^n$.
\end{theorem}

The proof of this theorem uses the recent characterization of the blow-up boundary behavior of solutions to linear fractional elliptic problems of type \eqref{radial pb}, given in~\cite{DLFeVi} (we need the assumption that $V \in W^{1,p}(\Omega)$ precisely in order to exploit the main results in \cite{DLFeVi}). In fact, we show that under the assumptions of Theorem~\ref{thm: not inf many zeros} the blow-up limit around any boundary point must vanish identically, and this is possible only if $u \equiv 0$ from the beginning. 

By Theorem \ref{thm: not inf many zeros}, we deduce that any non-trivial radial solution to the linear problem \eqref{radial pb} cannot vanish infinitely many times in a neighborhood of a boundary point. Therefore, if $x_0 \in \pa B_\rho$, then either~$u>0$ or~$u<0$ in~$B_\rho \cap B_R(x_0)$, for some $R>0$. Moreover, by Theorem \ref{thm: hopf}, the following alternative holds:
\begin{itemize}
\item[($i$)] either $u$ grows faster than $2s$ at the boundary, and then
\[
\frac{u}{d^s}(x_0) \neq 0 \quad \text{at any boundary point $x_0$}.
\]
\item[($ii$)] or $u$ does not grow faster than $2s$ at the boundary, and then
\[
\frac{u}{d^s}(x_0) = 0 \quad \text{at any boundary point $x_0$}.
\]
\end{itemize}

It is interesting to investigate under which conditions phenomena as in alternative ($ii$) take place.
With respect to this,
by leveraging some explicit representations and careful expansions at boundary and interior points, we give a negative answer to Question~\ref{LADOMA2}, according to the following result:

\begin{theorem}\label{INU10}
Let~$\rho>0$. Then, there exist~$V\in C^\infty_c(B_\rho)$ and~$u \in C^s(\R^n)$ satisfying
\begin{equation}\label{AKMSweFSINNSwerfmeacL}
\begin{cases}
(-\Delta)^s u = V(x) u & \text{in $B_\rho$}, \\
u = 0 & \text{in $\R^n \setminus B_\rho$},
\end{cases}
\end{equation}
such that $u$ does not vanish identically and does not change sign in a neighborhood of~$\rho e_1$ in~$B_\rho$, with
$$ \lim_{B_\rho\ni x\to\rho e_1}\frac{u}{d^s}(x)=0.$$
Also, $u$ is rotationally symmetric. 
\end{theorem}

For completeness, we recall that the radial eigenvalue problem (which corresponds to~$V(x) = const.$ in~\eqref{radial pb}) presents a different structure
and in this case we note that~$\frac{u}{d^s}$ never vanishes along the boundary
of the ball, as expressed by the following result:

\begin{corollary}\label{eigenv}
Let $\rho>0$ and $s \in (0,1)$. Let $u \in C^s(\R^n) \cap L^\infty(\R^n)$ be a radial eigenfunction of the fractional Laplacian in~$B_\rho$ with homogeneous Dirichlet exterior datum, namely~$u \not \equiv 0$ is such that
\begin{equation}\label{ev pb}
\begin{cases}
(-\Delta)^s u = \lambda u & \text{in $B_\rho$}, \\
u = 0 & \text{in $\R^n \setminus B_\rho$},
\end{cases}
\end{equation}
for some $\lambda>0$. Then
\[ \lim_{B_\rho\ni x\to x_0}
\frac{u}{d^s}(x) \neq 0 \quad \text{at any boundary point $x_0 \in \pa B_\rho$}.
\]
\end{corollary}

This corollary follows directly from the Pohozaev identity for the fractional Laplacian \cite{RSSe}, and crucially relies on the specific form of the right hand side in \eqref{ev pb}.

\medskip

The paper is organized as follows. In Section~\ref{sec: hopf} we provide the proof of Theorem~\ref{thm: hopf}, Section~\ref{sec3ugrihvrkjegi}
is devoted to the proof of Theorem~\ref{thm: hopf:NO} and Section~\ref{sec: non ex} to the proof of
Theorem~\ref{thm: not inf many zeros}. 
The proof of Theorem~\ref{INU10} is contained in Section~\ref{INU10S}.

\section{Proof of the Hopf Lemma in Theorem~\ref{thm: hopf}}\label{sec: hopf}

Here we prove Theorem~\ref{thm: hopf}.
For this,
we recall from \cite[Lemma 5.4]{RosOtsur} that whenever $L$ is a stable operator of the form~\eqref{def L}-\eqref{def a}, the function
\begin{equation}\label{def v}
v(x) = (1-|x|^2)_+^s
\end{equation}
is a solution to
\[
\begin{dcases}
L v=k & \text{in $B_1$,} \\
v = 0 & \text{in $\R^n \setminus B_1$},
\end{dcases}
\]
for some positive constant $k$. This allows us to prove that:

\begin{lemma}\label{lem: sub}
There exist a nonnegative function $\varphi \in C^s(\R^n)$ and a constant $C>0$ 
depending on~$n$, $s$, $\lambda$ and~$\Lambda$
such that
\[
\begin{dcases}
L \varphi \le -1 & \text{in }B_1 \setminus \overline{B_{1/2}}, \\
\varphi \ge (1-|x|^2)_+^s & \text{in $B_1$,} \\
\varphi \le C & \text{in }B_{1/2}, \\
\varphi = 0 & \text{in }\R^n \setminus B_1.
\end{dcases}
\]
\end{lemma}
\begin{proof}
Let $\eta \in C^\infty_c(B_{1/4})$ be a nonnegative function with $\int_{\R^n} \eta = 1$. For $x \in B_1 \setminus \overline{B_{1/2}}$, we have that
\begin{eqnarray*}
L \eta(x) = -PV \int_{\R^n} \frac{\eta(x+y)}{|y|^{n+2s}} a\left(\frac{y}{|y|}\right)\,dy \le -\lambda \int_{ B_{1/4}(-x)} \frac{\eta(x+y)}{|y|^{n+2s}}\,dy,
\end{eqnarray*}
thanks to~\eqref{def a}. Notice that if~$y\in B_{1/4}(-x)$ then~$|y|\le |x+y|+|x|\le \frac14 +1=\frac54$, and thus
$$ L \eta(x) \le -\lambda \left(\frac45\right)^{n+2s} \int_{ B_{1/4}(-x)}\eta(x+y)\,dy
=-\lambda \left(\frac45\right)^{n+2s}\int_{B_{1/4}} \eta(z)\,dz=-\lambda \left(\frac45\right)^{n+2s}.$$
Therefore, the function $\varphi:= v+C_1 \eta$ fulfills all the desired requirements, for $C_1>0$ sufficiently large. 
\end{proof}

Now we present an integral computation of general use, to be employed in the proof of Theorem~\ref{thm: hopf}.

\begin{lemma}\label{LE9COST-01}
Let $\Omega$ be an open subset of~$\R^n$, and let~$x_0\in\partial\Omega$.
Let $L$ be any operator of the form~\eqref{def L}-\eqref{dieur4bvi}-\eqref{def a}, and let $u:\R^n \to \R$ be a continuous function such that
\begin{equation}\label{1.3BIS}
u^-(y) \le \bar C\big(1+|y-x_0|^{2s-\delta}\big) \quad {\mbox{ for all }} y \in \R^n,
\end{equation}
for some~$\bar C>$ and~$\delta\in(0,2s)$. 

Suppose that $u \ge 0$ in a ball $B_R(x_0)$.
Let also~${\mathcal{S}}\Subset B_R(x_0)$
and~$d_0$ denote the distance between~${\mathcal{S}}$ and~$\partial B_R(x_0)$.

Then, there exists~$\widetilde C>0$,
depending on~$n$,
$s$, $\Lambda$, $R$ and $\bar C$ such that,
for every~$x\in{\mathcal{S}}$,
\begin{equation*}
-L u^-(x)\le
\widetilde C\left(\frac1{d_0^{2s}}+\frac1{d_0^\delta}\right).
\end{equation*}
\end{lemma}

\begin{proof} Since~$x\in{\mathcal{S}}\subset B_R(x_0)\subseteq\{u\ge0\}$, we have that~$u^-(x)=0$.

Furthermore, if~$y\in\R^n$ is such that~$u^-(x+y)>0$, then necessarily~$x+y\in\R^n\setminus B_R(x_0)$.
Hence, $|y|=|(x+y)-x|\ge d_0$. 

Consequently, making use of~\eqref{def a} and~\eqref{1.3BIS},
\begin{eqnarray*} -L u^-(x)&=&\int_{\R^n} \frac{u^-(x+y)}{|y|^{n+2s}} a\left(\frac{y}{|y|}\right)\,dy\\ &\le&
\int_{\R^n\setminus B_{d_0}}\frac{u^-(x+y)}{|y|^{n+2s}} a\left(\frac{y}{|y|}\right)\,dy \\&\le&
\bar C\Lambda\int_{\R^n\setminus B_{d_0}}\frac{1+|x+y-x_0|^{2s-\delta}}{|y|^{n+2s}}\,dy \\&\le&
\bar C\Lambda\int_{\R^n\setminus B_{d_0}}\frac{1+(|x-x_0|+|y|)^{2s-\delta}}{|y|^{n+2s}}\,dy \\&\le&
\bar C\Lambda\int_{\R^n\setminus B_{d_0}}\frac{1+(R+|y|)^{2s-\delta}}{|y|^{n+2s}}\,dy\\&\le&
2^{2s}\bar C\Lambda\int_{\R^n\setminus B_{d_0}}\frac{1+R^{2s-\delta}+|y|^{2s-\delta}}{|y|^{n+2s}}\,dy,
\end{eqnarray*}
from which the desired result follows.
\end{proof}

\begin{proof}[Proof of Theorem~\ref{thm: hopf}]
Under the assumptions of the theorem, there exists $\bar r \in (0,R/4)$ such that, for every~$r \in (0,\bar r]$, there exists an interior ball $B_r(x_r)$, tangent to $\pa \Omega$ at $x_0$, and such that $B_r(x_r) \subset B_R(x_0) \cap \Omega$. Let~$C$ and $\varphi$ be as in Lemma~\ref{lem: sub}. Also, set
\[
\alpha_r:=\frac1C\,\inf_{B_{r/2}(x_r)} u \qquad{\mbox{and}}\qquad
\psi(x) := \alpha_r \,\varphi\left(\frac{x-x_r}{r}\right).
\]
We observe that, since $u >0$ in $B_R(x_0) \cap \Omega \Supset B_{r/2}(x_r)$, we know that~$\alpha_r>0$.
Then, by Lemma~\ref{lem: sub},
\[
\begin{dcases}
L \psi \le -\frac{\alpha_r}{r^{2s}} & \text{in }B_r(x_r) \setminus \overline{B_{r/2}(x_r)}, \\
\psi \ge \frac{\alpha_r}{r^{2s}} (r^2-|x-x_r|^2)_+^s & \text{in $B_r(x_r)$,} \\
\psi \le \alpha_r C & \text{in }B_{r/2}(x_r), \\
\psi = 0 & \text{in }\R^n \setminus B_r(x_r).
\end{dcases}
\]
Let $w:= \psi-u^-$. Then, clearly, $w \le u$ in $\R^n \setminus B_r(x_r)$. Moreover, in $B_{r/2}(x_r)$, 
$$w=\psi\le \alpha_r C =\inf_{B_{r/2}(x_r)} u
\le u .$$

We now utilize Lemma~\ref{LE9COST-01}
with~${\mathcal{S}}:= B_r(x_r) \setminus \overline{B_{r/2}(x_r)}$.
To this end, we point out that the distance between~$B_r(x_r)$ and~$\partial B_R(x_0)$
is larger than~$R/2$, as long as~$r$ is sufficiently small with respect to~$R$.
Therefore, by Lemma~\ref{LE9COST-01}, for every~$x \in B_r(x_r) \setminus \overline{B_{r/2}(x_r)}$,
$$ -L u^-(x)\le C_\star,$$
with~$C_\star>0$ independent of $r$. Accordingly, for every~$x \in B_r(x_r) \setminus \overline{B_{r/2}(x_r)}$,
\begin{align*} L w(x) = L \psi(x)- L u^-(x)\le  -\frac{\alpha_r}{r^{2s}} + C_{\star}.
\end{align*}
Therefore, for all~$x\in B_r(x_r) \setminus \overline{B_{r/2}(x_r)}$, we have that
\begin{eqnarray*}&&
L(u-w)(x) \ge V(x)u(x)-Lw(x)\ge
\frac{\alpha_r}{r^{2s}} -C_\star  -V^-(x) u^+(x)    \\&&\qquad\ge \frac{\alpha_r}{r^{2s}} -C_\star - \|V^- u^+\|_{L^\infty(B_R(x_0))}
=\frac{\Phi(r)}{C} -C_\star - \|V^- u^+\|_{L^\infty(B_R(x_0))},
\end{eqnarray*} where~$\Phi$ is as in~\eqref{OLOms}.
As a consequence, by~\eqref{OLOms},
we can choose~$r>0$ sufficiently small, such that~$ L(u-w)(x)\ge0$
for all~$x\in B_r(x_r) \setminus \overline{B_{r/2}(x_r)}$.

Summarizing, we have that, for~$r$ small enough (that is now given once and for all), 
\[
\begin{dcases}
L(u-w) \ge 0 & \text{in }B_r(x_r) \setminus \overline{B_{r/2}(x_r)} ,\\
u-w \ge 0 & \text{in }\R^n \setminus \big(B_r(x_r) \setminus \overline{B_{r/2}(x_r)}\big),
\end{dcases}
\]
and the comparison principle yields that 
\begin{equation}\label{PI:0-1}
{\mbox{$u \ge w = \psi$ in $B_r(x_r) \setminus \overline{B_{r/2}(x_r)}$.}}\end{equation}

We can thus complete the proof of Theorem~\ref{thm: hopf} by combining~\eqref{PI:0-1}
and a geometric argument, reasoning as follows. We let~$\bar\nu$
be the vector joining~$x_0$ to the center of the interior sphere (note that~$\bar\nu$ is pointing ``inward'' the domain~$\Omega$).
Given~$\beta\in(0,\pi/2)$, we consider the set~${\mathcal{C}}_\beta$ of points~$x$ for
which the angle between~$x-x_0$ and~$\bar\nu$ is smaller than~$\pi/2-\beta$, that is
\begin{equation}\label{CBETA} {\mathcal{C}}_\beta:=\left\{ x\in\Omega{\mbox{ s.t. }} \frac{x-x_0}{|x-x_0|}\cdot\bar\nu>c_\beta
\right\},\qquad{\mbox{where }}\,c_\beta:=\cos\left(\frac\pi2-\beta\right)>0.\end{equation}
With this notation, to prove~\eqref{HOP-0ok} we have to check that
\begin{equation}\label{HOP-0ok-inp}
\liminf_{{\mathcal{C}}_\beta\ni x \to x_0} \frac{u(x)}{|x-x_0|^s} \ge c,
\end{equation}
for an appropriate~$c>0$ as in the statement of Theorem~\ref{thm: hopf}.

Thus, to establish~\eqref{HOP-0ok-inp}, we take a sequence of points~$x_k\in{\mathcal{C}}_\beta$ such that~$x_k\to x_0$
as~$k\to+\infty$ and we claim that, for~$k$ large enough,
\begin{equation}\label{HOP-0ok-inp2}
x_k\in B_{r}(x_{r})\setminus \overline{B_{r/2}(x_{r})}.
\end{equation}
Indeed, on the one hand,
\begin{eqnarray*}
&&|x_k-x_r|^2=\big|x_k-(x_0+r\bar\nu)\big|^2=
\big|(x_k-x_0)-r\bar\nu\big|^2
= |x_k-x_0|^2+r^2-2r(x_k-x_0)\cdot\bar\nu\\&&\qquad
=r^2-|x_k-x_0| \left(\frac{2r(x_k-x_0)\cdot\bar\nu}{|x_k-x_0|}-|x_k-x_0|\right)\le
r^2-|x_k-x_0| \big(2 c_\beta r-|x_k-x_0|\big)\\&&\qquad\le
r^2-c_\beta r|x_k-x_0|<r^2,
\end{eqnarray*}
as long as~$k$ is sufficiently large,
whence~$x_k\in B_{r}(x_{r})$.

On the other hand,
\begin{eqnarray*}
&&|x_k-x_{r}|\ge |x_0-x_r|-|x_k-x_0|=r-|x_k-x_0|>\frac{r}2,
\end{eqnarray*}
as long as~$k$ is large enough,
therefore~$x_k\not\in\overline{B_{r/2}(x_{r})}$.

The proof of~\eqref{HOP-0ok-inp2} is thereby complete.

Now, owing to~\eqref{HOP-0ok-inp2}, we are in the position of applying~\eqref{PI:0-1}, thus finding that
\begin{eqnarray*}&&
u (x_k)\ge \psi(x_k)\ge\frac{\alpha_r}{r^{2s}} \big(r^2-|x_k-x_{r}|^2\big)_+^s=
\frac{\alpha_r}{r^{2s}} \big(r^2-\big|(x_k-x_0)-r\bar\nu\big|^2\big)_+^s\\&&\qquad
=\frac{\alpha_r}{r^{2s}} \big(
2r(x_k-x_0)\cdot\bar\nu-|x_k-x_0|^2\big)_+^s\ge
\frac{\alpha_r}{r^{2s}} \big(
2c_\beta r|x_k-x_0|-|x_k-x_0|^2\big)_+^s
\end{eqnarray*}
and consequently
$$ \liminf_{k\to+\infty}
\frac{u(x_k)}{|x_k-x_0|^s}\ge
\liminf_{k\to+\infty}\frac{\alpha_r}{r^{2s}} \big(
2c_\beta r-|x_k-x_0|\big)_+^s=\frac{2^s\alpha_r\, c_\beta^s}{r^{s}}.$$
This completes the proof of~\eqref{HOP-0ok-inp}, as desired.
\end{proof}

\section{The necessity of additional conditions and proof of Theorem~\ref{thm: hopf:NO}}\label{sec3ugrihvrkjegi}

\begin{proof}[Proof of Theorem~\ref{thm: hopf:NO}]
We exploit~\cite[Theorem~3.1]{MR3626547} and we find~$r_2>r_1>0$ and a continuous function~$v$
such that
\begin{eqnarray*}
&& (-\Delta)^s v=0 {\mbox{ in }}B_{r_1},\\
&& v=0{\mbox{ in }}\R^n\setminus B_{r_2},\\
&& D^\gamma v(0)=0 {\mbox{ for all $\gamma=(\gamma_1,\dots,\gamma_n)\in\N^n$ such that~$|\gamma|:=\gamma_1+\dots+\gamma_n\le k$
and~$\gamma_1\ne k$}}\\ {\mbox{and }}
&& \partial^k_{x_1} v(0)=k!.\end{eqnarray*}
We take~$\mu:= r_1/4$ and define
$$ w(x):=\frac{v(\mu( x-x_0))}{\mu^k}.$$
Note that~$w$ is compactly supported, since so is~$v$. Additionally, if~$x\in B_2$ then~$|\mu( x-x_0)|\le \mu|x|+\mu|x_0|<3\mu\le r_1$
and therefore, for all~$x\in B_2$,
$$(-\Delta)^s w(x)=\mu^{2s-k}(-\Delta)^s v(\mu( x-x_0))=0 .$$

Also, $D^\gamma w(x)=\mu^{|\gamma|-k}D^\gamma v(\mu( x-x_0))$, from which we obtain that~$D^\gamma w(x_0)=0$ for all $\gamma\in\N^n$ such that~$|\gamma|\le k$
and~$\gamma_1\ne k$ and~$\partial^k_{x_1} w(x_0)=k!$.
Hence, a Taylor expansion gives that,
in a neighborhood of~$x_0$,
\begin{eqnarray*}
w(x)= \big((x-x_0)\cdot e_1\big)^k +O\big(|x-x_0|^{k+1}\big)
.\end{eqnarray*}
Now, for each~$j\in\{1,\dots,n\}$ we denote by~$w_j$ the function obtained by the previous procedure
when the variable~$x_1$ is replaced by~$x_j$ (of course, $w_1=w$) and we set~$u:=\frac{w_1+\dots+w_n}{n}$.
In this way, we have that~$(-\Delta)^s u=0$ in~$B_2$, $u$ is compactly supported and
\begin{eqnarray*}
u(x)=\sum_{j=1}^n\big((x-x_0)\cdot e_j\big)^k +O\big(|x-x_0|^{k+1}\big)
.\end{eqnarray*}
In particular, recalling that~$k$ is even (i.e., $k=2\ell$, with~$\ell\in\N\cap[1,+\infty)$), if~$\varrho>0$ is small enough,
$t\in(0,\varrho]$ and~$\omega\in\partial B_1$ then
\begin{eqnarray*}&& u(x_0+t\omega)=
\sum_{j=1}^n(t\omega_j)^k +O(t^{k+1})=
t^{2\ell} \left( \sum_{j=1}^n \omega_j^{2\ell} +O(t)\right)\ge
\frac{t^{2\ell}}{n^{\ell}}\left( \left(\sum_{j=1}^n \omega_j^{2}\right)^\ell +O(t)\right)\\&&\qquad\qquad\qquad\qquad\ge
\frac{t^{2\ell}}{n^{\ell}}\left( 1+O(\varrho)\right)\ge\frac{t^{2\ell}}{2n^{\ell}}>0.
\end{eqnarray*}
Besides,
$$ \limsup_{x\to x_0}\frac{u(x)}{|x-x_0|^k}=\limsup_{x\to x_0}\sum_{j=1}^n\frac{\big((x-x_0)\cdot e_j\big)^k}{|x-x_0|^k}
=\limsup_{y=(y_1,\dots,y_n)\to0}\frac{y_1^k+\dots+y_n^k}{|y|^k}=1$$
and the proof of Theorem~\ref{thm: hopf:NO} is thereby complete.
\end{proof}

\section{Radial solutions for the Dirichlet problem and proof of Theorem~\ref{thm: not inf many zeros}}\label{sec: non ex}

\begin{proof}[Proof of Theorem~\ref{thm: not inf many zeros}]
Without loss of generality, we can suppose that $\rho=1$.
 
In what follows, we will denote by~$\R^{n+1}_+:=\R^n \times (0,+\infty)$ and points~$z=(x,y)\in\R^{n+1}_+$. When needed, we will also
use the notation $z=(x_1,x',y) \in \R \times \R^{n-1} \times (0,+\infty)$ for points in~$\R^{n+1}_+$.
Additionally, balls centered at the origin with radius~$r$ in~$\R^{n+1}$ will be denoted by~${\mathcal{B}}_r$ and
$ {\mathcal{B}}_r^+:={\mathcal{B}}_r\cap \R^{n+1}_+$.

We argue by contradiction and suppose that~\eqref{radial pb} has a non-trivial radial bounded solution $u(x) = \tilde u(|x|)$, with infinitely many interior zeros accumulating at $1$: namely, we suppose that there exists a
sequence~$\{\rho_m\} \subset \R^+$ such that $\rho_m \to 1^-$ and $u|_{\pa B_{\rho_m}} = 0$ for every~$m$. It is convenient to rewrite~$\rho_m= (1+r_m)/2$, with~$r_m \to 1^-$.

Let us consider now the Caffarelli-Silvestre extension of $u(x)$, denoted by $U(z)$, with~$z=(x,y)\in\R^{n+1}_+$.
Since~$u$ is a solution of~\eqref{radial pb}, exploiting~\cite{CafSil}, we see that $U$ solves
\[
\begin{dcases}
\mathrm{div}(y^{1-2s}\nabla U) = 0 & \text{in $\R^{n+1}_+$,} \\
U = u & \text{on $\R^n \ \times \{0\}$,}\\
U = 0 & \text{on $(\R^n \setminus B_1) \times \{0\}$,}\\
-\displaystyle\lim_{y \to 0^+} y^{1-2s} \pa_y U = \kappa_{n,s} V(x) u & \text{on $B_1 \times \{0\}$,}
\end{dcases}
\]
where $\kappa_{n,s}$ is a positive constant depending on $n$ and $s$.

Let $x_0 =(1,0') \in \pa B_1$ and $z_0=(x_0,0)=(1,0',0)$. Then, by \cite[Theorem 1.3]{DLFeVi}, there exists~$k_0 \in \N$ such that the blow-up sequence
\[
U_m(z):= \frac{U(z_0+\lambda_m z)}{\lambda_m^{k_0+s}}, \quad \text{where} \quad \lambda_m := 1-r_m \to 0^+, 
\]
converges to a homogeneous \emph{non-trivial} function $\Psi(z) = |z|^{k_0+s} Y(z/|z|) \not \equiv 0$, in the weighted Sobolev space~$H^1({\mathcal{B}}_1^+, y^{1-2s}\,dz)$, as $m \to +\infty$. More precisely, $\Psi$ satisfies
\[
\begin{dcases}
\mathrm{div}(y^{1-2s}\nabla \Psi) = 0 & \text{in $\R^{n+1}_+$,} \\
\Psi = 0 & \text{in $\Gamma^+:= [0,+\infty) \times \R^{n-1} \times \{0\}$,}\\
\displaystyle\lim_{y \to 0^+} y^{1-2s} \pa_y \Psi =0 & \text{in $\Gamma^-:= (-\infty,0) \times \R^{n-1} \times \{0\}$,}
\end{dcases}
\]
see \cite[Appendix B]{DLFeVi}. 

We claim that
\begin{equation}\label{0702}
\Psi\left(-\frac12,x',0\right) = 0 \quad \text{whenever }|x'|<\frac14.
\end{equation}
To prove this claim, we recall that, by definition of~$r_m$ and~$\lambda_m$, the scaled function~$U_m|_{\{y=0\}}$ vanishes on the scaled sphere
\[
\frac{\partial B_{\frac{1+r_m}2}-(1,0')}{1-r_m} = \pa B_{\frac{1+r_m}{2(1-r_m)}}\left(\left(-\frac1{1-r_m},0'\right)\right) =: S_m.
\]

Now, if $x \in \{x_1=-1/2, |x'|<1/4\}$, then
\begin{equation}\label{addht985679yjhitrhjpotrjhotr}
{\mbox{there exists a sequence~$y_m^x \to x$ as~$m\to+\infty$, with~$y_m^x \in S_m$ for every~$m$.}}
\end{equation}
Indeed, given~$x=(-1/2,x')$ with~$|x'|<1/4$, we set~$y_m^x:=(y_{1,m}, x')$ with
\begin{align*}
y_{1,m} := \left(\frac{1}{1-r_m}\right) \left[ \frac{1+r_m}{2} \sqrt{1-\left(\frac{2|x'|(1-r_m)}{1+r_m}\right)^2}-1\right].\end{align*}
Notice that~$y_{1,m}$ is well-defined, since
$$ \frac{2|x'|(1-r_m)}{1+r_m}< \frac{1-r_m}{2(1+r_m)}\le\frac{1}{2}.
$$
We observe that~$y_m^x\in S_m$, since
\begin{eqnarray*}&&
\left| y_m^x-\left(-\frac1{1-r_m},0'\right)\right|=\left|\left(
\frac{1+r_m}{2(1-r_m)} \sqrt{1-\left(\frac{2|x'|(1-r_m)}{1+r_m}\right)^2}, x'\right)\right|\\&&\qquad\qquad\quad
=\sqrt{ \left(\frac{1+r_m}{2(1-r_m)}\right)^2\left( 1-\left(\frac{2|x'|(1-r_m)}{1+r_m}\right)^2\right)+|x'|^2}
=\frac{1+r_m}{2(1-r_m)}.
\end{eqnarray*}
Furthermore, we have that~$y_m^x\to x$ as~$m\to+\infty$, since
\begin{eqnarray*}
\left|y_{1,m}+\frac12\right| &=& \left|\left(\frac{1}{1-r_m}\right) \left[ \frac{1+r_m}{2} \sqrt{1-\left(\frac{2|x'|(1-r_m)}{1+r_m}\right)^2}-1\right]+\frac12\right|\\&=&\frac{1+r_m}2\cdot \frac{1-\sqrt{1-\left(\frac{2|x'|(1-r_m)}{1+r_m}\right)^2}}{1-r_m}\\&=&
\frac{ 2|x'|^2(1-r_m)}{(1+r_m)\left(1+\sqrt{1-\left(\frac{2|x'|(1-r_m)}{1+r_m}\right)^2}\right)},
\end{eqnarray*}
which converges to~$0$ as~$m\to+\infty$. This completes the proof of~\eqref{addht985679yjhitrhjpotrjhotr}.

Now, note that each $U_m$ solves 
\[
\begin{dcases}
\mathrm{div}(y^{1-2s}\nabla U_m) = 0 & \text{in $\R^{n+1}_+$,} \\
U_m = 0 & \text{on $\left(\R^n \setminus B_{1/\lambda_m}\left(\left(-\frac1{\lambda_m},0'\right)\right) \right)\times \{0\}$,}\\
-\displaystyle\lim_{y \to 0^+} y^{1-2s} \pa_y U_m = \kappa_s \lambda_m^{2s} V(x_0+\lambda_m x ) U_m(x,0) & \text{on $B_{1/\lambda_m}\left(\left(-\frac1{\lambda_m},0'\right)\right) \times \{0\}$.}
\end{dcases}
\]
Moreover, if~$x\in B_{1/3}\left(-\frac12,0'\right)$,
then
\begin{eqnarray*}
&&\left|x-\left(-\frac1{\lambda_m},0'\right)\right|\le
\left|x-\left(-\frac1{2},0'\right)\right|+\left|\left(\frac1{2},0'\right)-\left(\frac1{\lambda_m},0'\right)\right|
\le\frac13+\left|\frac12-\frac1{\lambda_m}\right|=\frac13+\frac1{\lambda_m}-\frac12<\frac1{\lambda_m},
\end{eqnarray*}
and thus
\[
B_{1/\lambda_m}\left(\left(-\frac1{\lambda_m},0'\right)\right) \times \{0\} \supset B_{1/3}\left( \left(-\frac12,0'\right)\right).
\]
Therefore, taking into account the boundedness of $U_m$ in $H^1({\mathcal{B}}_1^+, y^{1-2s}\,dz)$, Lemma~3.3 in~\cite{FaFe}
ensures that~$\|U_m\|_{C^\alpha(B_{1/3}( (-1/2,0')) \times [0,y_0])}\le C$, for some $y_0 >0$ and~$\alpha\in(0,1)$.
This in turn gives that~$U_m \to \Psi$ locally uniformly in~$B_{1/3}\left( \left(-\frac12,0'\right)\right) \times [0,y_0]$.

Furthermore, recalling~\eqref{addht985679yjhitrhjpotrjhotr}, if~$|x'|<1/4$ and~$m$ is sufficiently large,
\begin{eqnarray*}&&
\left|\Psi \left(-\frac12,x',0\right) - U_m\left(y_m^x,0\right)\right|\le
\left|\Psi \left(-\frac12,x',0\right) - U_m\left(-\frac12,x',0\right)\right|+\left|U_m \left(-\frac12,x',0\right) - U_m\left(y_m^x,0\right)\right|\\
&&\qquad \le\left|\Psi \left(-\frac12,x',0\right) - U_m\left(-\frac12,x',0\right)\right|+
\|U_m\|_{C^\alpha(B_{1/3}( (-1/2,0')) \times [0,y_0])}
\left|\left(-\frac12,x',0\right) - \left(y_m^x,0\right)\right|^\alpha\\
&&\qquad \le\left|\Psi \left(-\frac12,x',0\right) - U_m\left(-\frac12,x',0\right)\right|+
C\left|\left(-\frac12,x'\right) - y_m^x\right|^\alpha .
\end{eqnarray*}
{F}rom this, we obtain that, if~$|x'|<1/4$,
$$ \Psi \left(-\frac12,x',0\right) = \lim_{m} U_m\left(y_m^x,0\right).$$

Thus, since~$U_m|_{\{y=0\}}$ vanishes on the scaled sphere~$S_m$ and~$y_m^x\in S_m$, we conclude that
\[
\Psi \left(-\frac12,x',0\right) = \lim_{m} U_m\left(y_m^x,0\right) = 0 \qquad \text{for }|x'|<\frac14,
\]
which proves the claim in~\eqref{0702}. 

We now complete the proof of Theorem~\ref{thm: not inf many zeros} as follows.
Since~$\Psi$ is homogeneous with respect to $0$, equation~\eqref{0702} entails that $\Psi$ vanishes in an open cone of $\Gamma^-$, containing the ball $B_{1/4}(-1/2,0') \times \{0\} \subset \R^n \times \{0\}$. 

In particular,
\[
\begin{dcases}
\mathrm{div}(y^{1-2s}\nabla \Psi) = 0 & \text{in }\R^{n+1}_+ ,\\
\Psi = 0 & \text{on }B_{1/4}(-1/2,0') \times \{0\}, \\
\displaystyle\lim_{y \to 0^+} y^{1-2s} \pa_y \Psi = 0 & \text{on }B_{1/4}(-1/2,0') \times \{0\}.
\end{dcases}
\] 
By the unique continuation principle proved in \cite[Proposition 2.2]{Ru}, this gives that $\Psi \equiv 0$ in the extended ball ${\mathcal{B}}_{1/4}^+(-1/2,0',0) $. Therefore, the standard unique continuation principle gives that~$\Psi \equiv 0$ in~$\R^{n+1}_+$. This is the desired contradiction, and the proof is thereby complete.
\end{proof}

\section{Proof of Theorem~\ref{INU10}}\label{INU10S}

This part is devoted to the proof of Theorem~\ref{INU10},
relying also on the following boundary expansion of the fractional Green function of the ball
(see e.g.~\cite[Lemma~6]{MR3935264} for related results).

\begin{lemma} \label{TAYGRE} Let~$e\in\partial B_1$.
Let~$G$ be the fractional Green function of the ball of radius~$\rho$. Let~$\rho_0\in(0,\rho)$.

Then, for~$z\in B_{\rho_0}$, for small~$\delta>0$ we have that
$$ G((\rho-\delta) e,z)=a_0(z,e)\,\delta^s+o(\delta^{s}),$$
where
\begin{equation}\label{qoj1sw1dk2n234}\begin{split}
a_0(z,e):=\frac{2^s(\rho^2-|z|^2)^s}{s\rho^s\,|\rho e-z|^n},\end{split}\end{equation}
up to normalization constants that we omit, where~$o(\delta^{s})$ is uniform in~$z\in B_{\rho_0}$.
\end{lemma}

We also recall the interior blow-up behavior of the fractional Green function:

\begin{lemma}\label{Nmm-POJL} Let~$n\ne2s$.
For every~$e$, $y\in\overline B_1$,
$$ \lim_{\rho_0\searrow0} \rho_0^{n-2s} G(\rho_0 e,\rho_0 y)=|e-y|^{2s-n},$$
up to a dimensional constant that we omit.
\end{lemma}

For the reader's facility, we postpone the technical proofs of Lemmata~\ref{TAYGRE} and~\ref{Nmm-POJL} to Appendix~\ref{TAYGRE-A}.
Now we prove Theorem~\ref{INU10}.

\begin{proof}[Proof of Theorem~\ref{INU10}]
Let~$\e\in(0,1)$, to be taken appropriately small in what follows.
Let also~$\rho_0\in(0,\rho)$ and~$\phi\in C^\infty_c((-1,1),[0,+\infty))$ with~$\phi$ not identically zero.

Let
$$ \phi_{1,\e}(x):=\frac1{\e^n}\phi\left(\frac{|x|}\e\right)\qquad
{\mbox{and}}\qquad\phi_{2,\e}(x):=\frac1{\e}\phi\left(
\frac{|x|-\rho_0}{\e}\right).$$

We consider the solutions~$u_{1,\e}$ and~$u_{2,\e}$ of the problem
\begin{equation}\label{PKlms-1203o4rt}
\begin{cases}
(-\Delta)^s u_{j,\e} = \phi_{j,\e} & \text{in $B_\rho$}, \\
u_{j,\e} = 0 & \text{in $\R^n \setminus B_\rho$},
\end{cases}
\end{equation}
with~$j\in\{1,2\}$.

We remark that the functions~$u_{j,\e}$ are radial. To check this, one can observe that the Green function~$G(x,y)$ of the ball~$B_\rho$ is invariant under rotations (that is~$G(Rx,Ry)=G(x,y)$ for every rotation~$R$, as can be checked
from~\cite[Theorem~3.1]{MR3461641}) and then deduce that~$u_{j,\e}$ is radial due to the Green function representation (see~\cite[Theorem 3.2]{MR3461641}).

Using the Green function, one can also check that~$u_{j,\e}>0$ in~$B_\rho$. Also, by the fractional Hopf Lemma (see e.g.~\cite{RosOtsur}), we have that
$$ c_{j,\e}:=\lim_{B_\rho\ni x\to\partial B_\rho}\frac{u_{j,\e}}{d^s}(x)>0.$$
We need now to extract some more quantitative information and we will do so via Lemma~\ref{TAYGRE},
which entails that, for all~$e\in\partial B_1$,
\begin{equation*}\begin{split}
c_{1,\e}&=\lim_{\delta\searrow0}\frac{u_{1,\e}((\rho-\delta)e)}{\delta^s}\\&=\lim_{\delta\searrow0}\frac{1}{\delta^s}
\int_{B_\rho} \phi_{1,\e}(y)\,G((\rho-\delta)e,y)\,dy\\
&=\lim_{\delta\searrow0}\frac{1}{\delta^s\,\e^n}
\int_{B_\rho} \phi\left(\frac{|y|}\e\right)\,G((\rho-\delta)e,y)\,dy\\&=\lim_{\delta\searrow0}\frac{1}{\delta^s}
\int_{B_1} \phi(|w|)\,G((\rho-\delta)e,\e w)\,dw\\&=\lim_{\delta\searrow0}\frac{1}{\delta^s}
\int_{B_1} \phi(|w|)\,\Big( a_0(\e w,e)\,\delta^s+o(\delta^s)\Big)\,dw\\&=
\int_{B_1} \frac{2^s\,\phi(|w|)\,(\rho^2-\e^2|w|^2)^s}{s \rho^s\,|\rho e-\e w|^n}\,dw
\end{split}
\end{equation*}
and that
\begin{equation*}\begin{split}
c_{2,\e}&=\lim_{\delta\searrow0}\frac{u_{2,\e}((\rho-\delta)e)}{\delta^s}\\&=\lim_{\delta\searrow0}\frac{1}{\delta^s}
\int_{B_\rho} \phi_{2,\e}(y)\,G((\rho-\delta)e,y)\,dy\\
&=\lim_{\delta\searrow0}\frac{1}{\delta^s\,\e}
\int_{B_\rho} \phi\left(\frac{|y|-\rho_0}\e\right)\,G((\rho-\delta)e,y)\,dy\\&=\lim_{\delta\searrow0}\frac{1}{\delta^s\,\e}
\iint_{(\partial B_1)\times(0,\rho)} t^{n-1} \phi\left(\frac{t-\rho_0}\e\right)\,G((\rho-\delta)e,t\omega)\,dH^{n-1}_\omega\,dt\\&=\lim_{\delta\searrow0}\frac{1}{\delta^s}
\iint_{(\partial B_1)\times(-1,1)} (\rho_0+\e\tau)^{n-1} \phi(\tau)\,G((\rho-\delta)e,(\rho_0+\e\tau)\omega)\,dH^{n-1}_\omega\,d\tau\\&=\lim_{\delta\searrow0}\frac{1}{\delta^s}
\iint_{(\partial B_1)\times(-1,1)} (\rho_0+\e\tau)^{n-1} \phi(\tau)\,\Big(
a_0((\rho_0+\e\tau)\omega,e)\,\delta^s+o(\delta^s) \Big)\,dH^{n-1}_\omega\,d\tau\\&=
\iint_{(\partial B_1)\times(-1,1)} 
\frac{2^s\,(\rho_0+\e\tau)^{n-1} \,(\rho^2-|(\rho_0+\e\tau)\omega|^2)^s\,\phi(\tau)}{s\rho^s\,|\rho e-(\rho_0+\e\tau)\omega|^n}\,dH^{n-1}_\omega\,d\tau.
\end{split}
\end{equation*}

These observations entail that
\begin{equation}\label{C1skdmc} \lim_{\e\searrow0}c_{1,\e}=
\int_{B_1} \frac{2^s \phi(|w|)\,\rho ^{s-n}}s \,dw=:c_1
\end{equation}
and
\begin{equation}\label{C2skdmc} \lim_{\e\searrow0}c_{2,\e}=\iint_{(\partial B_1)\times(-1,1)} 
\frac{ 2^s\rho_0^{n-1} \,(\rho^2-\rho_0^2)^s\,\phi(\tau)}{s\rho^s\,|\rho e-\rho_0\omega|^{n}}
\,dH^{n-1}_\omega\,d\tau=:c_2.\end{equation}

We stress that~$c_1$, $c_2\in(0,+\infty)$ thus, assuming that~$\e$ is conveniently small, we can suppose that
\begin{equation*} c_{1,\e}\in\left(\frac{c_1}2,2c_1\right)\qquad{\mbox{and}}\qquad c_{2,\e}\in\left(\frac{c_2}2,2c_2\right).\end{equation*}

Now we define
$$ u_\e:= c_{2,\e} \, u_{1,\e}-c_{1,\e}\, u_{2,\e}.$$
In this way,
\begin{equation}\label{dj3i547v68cib7t5yghvnckmlsx}
\lim_{B_\rho\ni x\to\rho e_1}\frac{u_\e}{d^s}(x)=\lim_{B_\rho\ni x \to\rho e_1}\frac{c_{2,\e} \, u_{1,\e}(x)-c_{1,\e}\, u_{2,\e}(x)}{d^s(x)}=c_{2,\e} \, c_{1,\e}-c_{1,\e}\, c_{2,\e}=0.
\end{equation}

Furthermore, we claim that, if~$\e$ is small enough,
\begin{equation}\label{SUPPO-1}
\{u_\e=0\}\cap \supp \phi_{1,\e}=\varnothing
\end{equation}
and
\begin{equation}\label{SUPPO-2}
\{u_\e=0\}\cap \supp \phi_{2,\e}=\varnothing.
\end{equation}
We postpone the proofs of these claims at the end and we now complete the proof of Theorem~\ref{INU10},
by assuming the validity of these claims.

Thanks to~\eqref{SUPPO-1} and~\eqref{SUPPO-2}, we can define
$$ V_{j,\e}:=\frac{\phi_{j,\e}}{u_\e},$$
with the implicit notation that~$V_{j,\e}=0$ outside the support of~$\phi_{j,\e}$. 
Let also
$$ V_\e:=c_{2,\e}V_{1,\e}-c_{1,\e}V_{2,\e}.$$
Thus, it follows from~\eqref{PKlms-1203o4rt} that, in~$B_\rho$,
$$ (-\Delta)^s u_\e=c_{2,\e}\phi_{1,\e}-c_{1,\e}\phi_{2,\e}=(c_{2,\e}V_{1,\e}-c_{1,\e}V_{2,\e})u_\e=V_\e u_\e.$$
Additionally, $u_\e$ does not vanish identically, as a byproduct of~\eqref{SUPPO-1}, and it does not change sign in a neighborhood of~$\partial B_\rho$, owing to Theorem~\ref{thm: not inf many zeros}.
In this way, recalling also~\eqref{dj3i547v68cib7t5yghvnckmlsx},
the function~$u:=u_\e$ satisfies all the desired requirements in 
Theorem~\ref{INU10}.

It remains to prove~\eqref{SUPPO-1} and~\eqref{SUPPO-2}. For this,
we will use that
\begin{equation}\label{TIJSDttTJMS}
\begin{split}
u_\e(x)&=c_{2,\e} \, u_{1,\e}(x)-c_{1,\e}\, u_{2,\e}(x)\\&=\frac{
c_{2,\e}}{\e^n}\int_{B_\rho}\phi\left(\frac{|y|}\e\right)\,G(x,y)\,dy-\frac{c_{1,\e}}{\e}\int_{B_\rho}\phi\left(\frac{|y|-\rho_0}{\e}\right)\,G(x,y)\,dy\\&=
c_{2,\e}\int_{B_1}\phi(|w|)\,G(x,\e w)\,dw-
\frac{c_{1,\e}}{\e}\iint_{(\partial B_1)\times(0,\rho)}
t^{n-1} \phi\left(\frac{t-\rho_0}{\e}\right)\,G(x,t\omega)\,dH^{n-1}_\omega\,dt\\
&=c_{2,\e}\int_{B_1}\phi(|w|)\,G(x,\e w)\,dw\\&\qquad\qquad-
c_{1,\e}\iint_{(\partial B_1)\times(-1,1)}
(\rho_0+\e\tau)^{n-1} \phi(\tau)\,G(x,(\rho_0+\e\tau)\omega)\,dH^{n-1}_\omega\,d\tau.
\end{split}
\end{equation}
In the proof of~\eqref{SUPPO-1} and~\eqref{SUPPO-2},
we need to distinguish the case~$s\in\left(0,\frac12\right]$ from the case~$s\in\left(\frac12,1\right)$.\medskip

\noindent{\bf (1). The case~$s\in\left(0,\frac12\right]$.}
Let us start by proving~\eqref{SUPPO-1}. To this end, we argue by contradiction
and we suppose that there exists a point~$p_\e\in\{u_\e=0\}\cap \supp \phi_{1,\e}\subseteq \{u_\e=0\}\cap B_\e$. 

We recall that, for~$|X-Y|$ small (and~$X$ and~$Y$ away from~$\partial B_\rho$), we have (see e.g.~\cite[equations~(1.17) and~(1.19)]{MR3461641})
\begin{equation} \label{GREE:DI}G(X,Y)\ge \begin{cases}
c|X-Y|^{2s-n} & {\mbox{ if }}n>2s,\\
c |\log|X-Y||& {\mbox{ if }}n=2s.
\end{cases}\end{equation}
This and~\eqref{TIJSDttTJMS} give that, for small~$\e$,
\begin{equation}\label{PAKsjldnf201i3roufehg219poihwqfekbjvf98237rty832o9i4gth4epir}
\begin{split}
0&=u_\e(p_\e)\\&\ge
\frac{c_{2}}2\int_{B_1}\phi(|w|)\,G(p_\e,\e w)\,dw-\frac{2c_{1}\|\phi\|_{L^\infty((-1,1))}}{\e}\int_{B_{\rho_0+\e}\setminus B_{\rho_0-\e}}G(p_\e,y)\,dy\\&\ge
\frac{c_{2}}2\int_{B_1}\phi(|w|)\,dw
\,\inf_{X,Y\in B_\e}G(X,Y)-C.
\end{split}\end{equation}
Since, by~\eqref{GREE:DI}, the Green function~$G(X,Y)$ has a singularity when~$X=Y$ (in fact, here we are using only that~$n\ge2s$), it follows that the last term is as large as we wish, and in particular strictly positive.
This is a contradiction and~\eqref{SUPPO-1} is thereby established in this case.

Now we prove~\eqref{SUPPO-2}. For this, we suppose that
there exists a point~$q_\e\in\{u_\e=0\}\cap \supp \phi_{2,\e}\subseteq \{u_\e=0\}\cap B_{\rho_0+\e}\setminus B_{\rho_0-\e}$. 

Hence, recalling~\eqref{TIJSDttTJMS}, for small~$\e$, since~$u_\e$ is radial, for all~$e\in\partial B_1$ we have that
\begin{eqnarray*}
0&=&-u_\e(|q_\e|e)\\&\ge&
\frac{c_{1}}{2} \iint_{(\partial B_1)\times(-1,1)} (\rho_0+\e\tau)^{n-1}\phi(\tau)\,G(|q_\e|e,(\rho_0+\e\tau)\omega)\,dH^{n-1}_\omega\,d\tau
\\&&\qquad
-\frac{2c_{2}\,\|\phi\|_{L^\infty((-1,1))}}{\e^n}\int_{B_\e}\,G(|q_\e|e,y)\,dy\\&\ge&
\frac{c_{1}}{2} \iint_{(\partial B_1)\times(-1,1)} (\rho_0+\e\tau)^{n-1}\phi(\tau)\,G(|q_\e|e,(\rho_0+\e\tau)\omega)\,dH^{n-1}_\omega\,d\tau-C.
\end{eqnarray*}
The contradiction here is obtained\footnote{For simplicity, we wrote~\eqref{ojldnd-23e4r} when~$n\ge2$. Notice that we have used there that~$n-2s\ge n-1$. When~$n=1$, we just obtain
$$ \frac{c\,c_{1}}{2\rho_0^{1-2s}} \int_{-1}^1 \phi(\tau)\,d\tau |\log 0|-C=+\infty.$$}
by  considering an~$\eta\in(0,1)$ sufficiently small and noticing that the limit as~$\e\searrow0$ gives, up to a subsequence, and recalling~\eqref{GREE:DI},
\begin{equation}\label{ojldnd-23e4r}
\begin{split}
0&\ge
\frac{c_{1}\rho_0^{n-1}}{2} \int_{-1}^1 \phi(\tau)\,d\tau \int_{\partial B_1}G(\rho_0 e,\rho_0\omega)\,dH^{n-1}_\omega-C
\\&\ge\frac{c\,c_{1}}{2\rho_0^{1-2s}} \int_{-1}^1 \phi(\tau)\,d\tau \int_{\partial B_1\cap\{|e-\omega|<\eta\}}\,\frac{dH^{n-1}_\omega}{|e-\omega|^{n-2s}}-C\\&=+\infty.\end{split}
\end{equation}
This is a contradiction and the proof of~\eqref{SUPPO-2} is thereby complete in this case.\medskip

\noindent{\bf (2). The case~$s\in\left(\frac12,1\right)$.}
To prove~\eqref{SUPPO-1}, we observe that
the argument in~\eqref{PAKsjldnf201i3roufehg219poihwqfekbjvf98237rty832o9i4gth4epir} goes through whenever~$n\ge2s$ and therefore we focus here on the case~$n<2s$, which gives~$n=1$.
Thus, in this case we use~\eqref{TIJSDttTJMS} and we find that,
if there existed a point~$p_\e\in\{u_\e=0\}\cap \supp \phi_{1,\e}\subseteq \{u_\e=0\}\cap (-\e,\e)$, then 
\begin{eqnarray*}
0&=&
c_{2,\e}\int_{-1}^{1}\phi(|w|)\,G(p_\e,\e w)\,dw\\&&\qquad-
c_{1,\e}\left[\int_{-1}^1(\rho_0+\e\tau)^{n-1} \phi(\tau)\,G(p_\e,\rho_0+\e\tau)\,d\tau
+\int_{-1}^1(\rho_0+\e\tau)^{n-1} \phi(\tau)\,G(p_\e,-\rho_0-\e\tau)\,d\tau
\right].
\end{eqnarray*}
We now send~$\e\searrow0$ and utilize that~$G(0,0)=0$ in this case (see~\cite[equation~(1.17)]{MR3461641}), thus finding that
\begin{eqnarray*}
0=-
c_{1}\left[\int_{-1}^1\rho_0^{n-1} \phi(\tau)\,G(0,\rho_0)\,d\tau
+\int_{-1}^1 \rho_0^{n-1} \phi(\tau)\,G(0,-\rho_0)\,d\tau
\right]<0.
\end{eqnarray*}
This proves~\eqref{SUPPO-1} also in this case.

Now we turn to the proof of~\eqref{SUPPO-2}. Again, we suppose by contradiction that
there exists a point~$q_\e\in\{u_\e=0\}\cap \supp \phi_{2,\e}\subseteq \{u_\e=0\}\cap B_{\rho_0+\e}\setminus B_{\rho_0-\e}$ such that~$u(|q_\e|e)=u(q_\e)=0$, for all~$e\in\partial B_1$.

We recall~\eqref{TIJSDttTJMS} and we write that
\begin{eqnarray*}
0&=& c_{2,\e}\int_{B_1}\phi(|w|)\,G(|q_\e|e,\e w)\,dw\\&&\qquad-
c_{1,\e}\iint_{(\partial B_1)\times(-1,1)}
(\rho_0+\e\tau)^{n-1} \phi(\tau)\,G(|q_\e|e,(\rho_0+\e\tau)\omega)\,dH^{n-1}_\omega\,d\tau.
\end{eqnarray*}
Hence, sending~$\e\searrow0$,
\begin{eqnarray*}
c_{2} \int_{B_1}\phi(|w|)\,dw \,G(\rho_0 e,0)=
c_{1}\iint_{(\partial B_1)\times(-1,1)}
\rho_0^{n-1} \phi(\tau)\,G(\rho_0 e, \rho_0\omega)\,dH^{n-1}_\omega\,d\tau.
\end{eqnarray*}

On this account, substituting for~$c_1$ and~$c_2$ in light of~\eqref{C1skdmc} and~\eqref{C2skdmc}, after a simplification we find that
\begin{equation}\label{RHOSK}
\int_{\partial B_1} 
\frac{(\rho^2-\rho_0^2)^s}{|\rho e-\rho_0\omega|^{n}}
\,dH^{n-1}_\omega\,G(\rho_0 e,0)=\rho ^{2s-n}\int_{\partial B_1} G(\rho_0 e, \rho_0\omega)\,dH^{n-1}_\omega.
\end{equation}
Note that we are still free to modify~$\rho_0$ if needed.

Therefore, taking~$\rho_0$ as small as we wish, we infer from Lemma~\ref{Nmm-POJL} and~\eqref{RHOSK} (in the limit as~$\rho_0\searrow0$) that
\begin{equation}\label{O90}
H^{n-1}(\partial B_1)=\int_{\partial B_1} |e-\omega|^{2s-n}\,dH^{n-1}_\omega.
\end{equation}
When $n=1$, this boils down to
$$ 2= |1-1|^{2s-1}+|1+1|^{2s-1}=2^{2s-1},
$$
which is a contradiction.

Hence, we now deal with the case~$n\ge2$. Choosing~$e:=e_n$ in~\eqref{O90}, we obtain that
\begin{equation}\label{O901}\begin{split}
H^{n-1}(\partial B_1)&=\int_{\partial B_1} |e_n-\omega|^{2s-n}\,dH^{n-1}_\omega\\
&= \int_{\partial B_1\cap\{\omega_n>0\}} |e_n-\omega|^{2s-n}\,dH^{n-1}_\omega+\int_{\partial B_1\cap\{\omega_n<0\}} |e_n-\omega|^{2s-n}\,dH^{n-1}_\omega\\
&= \int_{\partial B_1\cap\{\omega_n>0\}} \Big(|e_n-\omega|^{2s-n}+ |e_n+\omega|^{2s-n}\Big)\,dH^{n-1}_\omega.
\end{split}\end{equation}

Now we look at the function
$$ (0,+\infty)\ni \sigma\mapsto F(\sigma):=\frac{1+\sigma^{2s-n}}{(1+\sigma)^{2s-n}}.$$
We note that
$$ \lim_{\sigma\to0}F(\sigma)=+\infty=\lim_{\sigma\to+\infty}F(\sigma)$$
and
$$ F'(\sigma)= \frac{n-2s}{ \sigma^{n+1-2s}} (1+\sigma)^{n - 2 s - 1} (\sigma^{n + 1-2s} - 1).$$
In particular, $F$ is strictly increasing when~$\sigma>1$ and strictly decreasing when~$\sigma\in(0,1)$, thus exhibiting a minimum at~$\sigma=1$, with~$F(1)=2^{n+1-2s}$.

As a byproduct, for every~$a$, $b>0$, if~$\sigma:=\frac{b}a$,
\begin{eqnarray*}&& a^{2s-n}+b^{2s-n}=a^{2s-n} (1+\sigma^{2s-n})=a^{2s-n}\,(1+\sigma)^{2s-n}\,F(\sigma)\\
&&\qquad\ge
a^{2s-n}\,(1+\sigma)^{2s-n}\,F(1)=2^{n+1-2s}(a+b)^{2s-n},\end{eqnarray*}
with strict inequality unless~$a=b$.

Therefore, choosing~$a:=|e_n-\omega|$ and~$b:=|e_n+\omega|$,
\begin{eqnarray*}&&
|e_n-\omega|^{2s-n}+ |e_n+\omega|^{2s-n}\ge 2^{n+1-2s}\Big(|e_n-\omega|+|e_n+\omega|\Big)^{2s-n}
\\&&\qquad\ge2^{n+1-2s}\Big(|(e_n-\omega)+(e_n+\omega)|\Big)^{2s-n}=
2^{n+1-2s} |2e_n|^{2s-n}=2,
\end{eqnarray*}
with strict inequality when~$\omega_n\ne0$.

For this reason,
$$ \int_{\partial B_1\cap\{\omega_n>0\}} \Big(|e_n-\omega|^{2s-n}+ |e_n+\omega|^{2s-n}\Big)\,dH^{n-1}_\omega> 2H^{n-1}(\partial B_1\cap\{\omega_n>0\})=H^{n-1}(\partial B_1),$$
which provides a contradiction with~\eqref{O901}. The proof of~\eqref{SUPPO-2} is thereby complete in this case as well.
\end{proof}

\begin{appendix}

\section{An interesting example}\label{CONTRO}

Here we observe that the condition~$V\in L^\infty(\Omega)$ in Corollary~\ref{CONTRO-0} cannot be removed, not even in the case of radial and positive solutions. To see this, let us consider
an even function~$u_0\in C(\R)$ such that~$u_0=0$ in~$[1,+\infty)$, $u_0\in C^\infty((-1,1))$ and~$u_0'<0$ in~$(0,1)$. In particular, we have that~$u_0>0$ in~$(-1,1)$. For every~$x\in\R^n$, we define~$u(x):=u_0(|x|)$.

For every~$x\in B_1$, we also define~$g(x):=(-\Delta)^su(x)$ and we stress that~$g\in C^\infty(B_1)$
(but it may become unbounded at~$\partial B_1$). Since~$u>0$ in~$B_1$, we can define, for all~$x\in B_1$,
$$ V(x):=\frac{g(x)}{u(x)}.$$
We stress that~$V\in L^{\infty}_{\loc}(B_1)$, but~$V$ is not necessarily in~$L^\infty(B_1)$,
and, by construction,
\[
\begin{cases}
(-\Delta)^s u=V(x) u & \text{in $B_1$}, \\
u=0 & \text{in $\R^n \setminus B_1$},
\end{cases} \quad \text{with} \quad u>0\; {\text{ in $\;B_1$}}. 
\]
However, we are free to choose any growth of~$u_0$ from the point~$1$ that we like and therefore
the result in Corollary~\ref{CONTRO-0} does not necessarily hold in this case.

\section{Proof of Lemmata~\ref{TAYGRE} and~\ref{Nmm-POJL}}\label{TAYGRE-A}

For completeness, we provide a self-contained proof of Lemmata~\ref{TAYGRE}
and~\ref{Nmm-POJL} about the boundary and interior behavior of the fractional
Green function for the ball. We will use suitable Taylor expansions (in fact, higher order expansions can be obtained similarly).

\begin{proof}[Proof of Lemma~\ref{TAYGRE}]
Up to a rotation, we can assume that~$e=e_1$.

First, we suppose that~$n\neq 2s$ and we let
\begin{equation}\label{JLSN-12weRo}
r_0(x,z):=\frac{(\rho^2-|x|^2)(\rho^2-|z|^2)}{\rho^2|x-z|^2}.\end{equation}
We observe that
\begin{eqnarray*}r_\delta&:=&
r_0((\rho-\delta)e_1,z)\\&=&\frac{(\rho^2-(\rho-\delta)^2)(\rho^2-|z|^2)}{\rho^2|(\rho-\delta)e_1-z|^2}
\\&=&\frac{(2\rho-\delta)(\rho^2-|z|^2)}{\rho^2(|\rho e_1-z|^2-2\delta\rho+2\delta z_1+\delta^2)}\,\delta
\\&=&\left(
1+\frac{2\delta\rho-2\delta z_1}{|\rho e_1-z|^2}+O(\delta^2)
\right)
\frac{(2\rho-\delta)(\rho^2-|z|^2)}{\rho^2\,|\rho e_1-z|^2}\,\delta\\&=&
b_1\delta+O(\delta^2),
\end{eqnarray*}
where
\begin{eqnarray*}
&&b_1:=\frac{2(\rho^2-|z|^2)}{\rho\,|\rho e_1-z|^2}.
\end{eqnarray*}

Now we point out that a Taylor expansion gives that, for small~$\sigma$,
$$ \frac{1}{(\sigma+1)^{\frac{n}2}}=1-\frac{n}2\,\sigma+O(\sigma^2)
$$
and therefore, for small~$\tau$,
$$ \frac{1}{(\tau^{\frac1s}+1)^{\frac{n}2}}=1-\frac{n}2\,\tau^{\frac1s}+O(\tau^{\frac2s}).$$
For this reason, for small~$\eta$,
\begin{eqnarray*} \int_0^\eta \frac{d\tau}{(\tau^{\frac1s}+1)^{\frac{n}2}}&=&\int_0^\eta\left[
1-\frac{n}2\,\tau^{\frac1s}+O(\tau^{\frac2s})\right]\,d\tau\\&=&
\eta-\frac{sn}{2(1+s)}\,\eta^{\frac{1+s}s}+O(\eta^{\frac{2+s}s}).\end{eqnarray*}
Hence, using the substitution~$\tau=t^s$, we conclude that
\begin{eqnarray*}
\int_0^{\eta^{\frac1s}}\frac{t^{s-1}}{(t+1)^{\frac{n}2}}\,dt=\frac1s\int_0^\eta \frac{d\tau}{(\tau^{\frac1s}+1)^{\frac{n}2}}=\frac\eta{s}-\frac{n}{2(1+s)}\,\eta^{\frac{1+s}s}+O(\eta^{\frac{2+s}s}).
\end{eqnarray*}
Substituting for~$\eta=r_\delta^s$, we find that
\begin{eqnarray*}
\int_0^{r_\delta}\frac{t^{s-1}}{(t+1)^{\frac{n}2}}\,dt&=&\frac{r_\delta^s}{s}-\frac{n}{2(1+s)}\,r_\delta^{1+s}+O(r_\delta^{2+s})\\&=&
\frac{1}{s}\big(b_1+O(\delta)\big)^s\,\delta^s+O(\delta^{1+s})\\&=&
\frac{b_1^s}{s}\,\delta^s+o(\delta^s).
\end{eqnarray*}

{F}rom this, using the explicit expression of the fractional Green function (see e.g.~\cite[equation~(1.17)]{MR3461641}), up to a normalizing constant that we neglect we have that
\begin{eqnarray*}&&
G((\rho-\delta)e_1,z)\\&=&|(\rho-\delta)e_1-z|^{2s-n}\left(\frac{b_1^s}{s}\,\delta^s+o(\delta^{s})\right)\\&=&
\Big( |\rho e_1-z|^2-2\delta\rho+2\delta z_1+\delta^2 \Big)^{\frac{2s-n}2}\left(\frac{b_1^s}{s}\,\delta^s+o(\delta^{s})\right)\\&=&
|\rho e_1-z|^{2s-n}
\left(1+O(\delta) \right)\left(\frac{b_1^s}{s}\,\delta^s+o(\delta^{s})\right)\\&=&
\frac{|\rho e_1-z|^{2s-n}\,b_1^s }{s}\,\delta^s
+o(\delta^{s}).
\end{eqnarray*}
This gives the desired result in this case, since
\begin{eqnarray*}
\frac{|\rho e_1-z|^{2s-n}\,b_1^s}{s}=
\frac{2^s(\rho^2-|z|^2)^s}{s\rho^s\,|\rho e_1-z|^{n}}=a_0(z,e_1).
\end{eqnarray*}

Having proven the desired result when~$n\neq 2s$, we now focus on the case~$n=2s$, which is~$n=1$ and~$s=\frac12$. Here, we will use the logarithmic representation of the fractional Green function in this case (see e.g.~\cite[equation~(1.19)]{MR3461641}).

To this end, we observe that
\begin{eqnarray*}
\sqrt{(\rho^2-(\rho-\delta)^2)(\rho^2-z^2)}&=&\sqrt{(2\delta\rho-\delta^2)(\rho^2-z^2)}\\
&=&
\sqrt{2\rho(\rho^2- z^2)}\,\delta^{\frac12} +O(\delta^{\frac32}).
\end{eqnarray*}
Besides,
\begin{eqnarray*}
\frac{1}{|z-(\rho-\delta)|}=
\frac{1}{\rho-z}+ O(\delta)
\end{eqnarray*}
and accordingly
\begin{eqnarray*}&&
\frac{\rho^2-(\rho-\delta)z+\sqrt{(\rho^2-(\rho-\delta)^2)(\rho^2-z^2)} }{\rho |z-(\rho-\delta)|}\\&=&
\left( \rho(\rho- z)+\delta z
+\sqrt{2\rho(\rho^2- z^2)}\,\delta^{\frac12}+O(\delta^{\frac32})\right)
\left( \frac{1}{\rho(\rho-z)}+ O(\delta)\right)\\&=&
1 + \frac{\sqrt{2 \rho(\rho^2-z^2)}}{\rho (\rho-z)} \,\delta^{\frac12}+ O(\delta).
\end{eqnarray*}
For this reason,
\begin{eqnarray*}&&G((\rho-\delta)e_1,z)\\&=&
\log\left(\frac{\rho^2-(\rho-\delta)z+\sqrt{(\rho^2-(\rho-\delta)^2)(\rho^2-z^2)} }{\rho |z-(\rho-\delta)|}\right)\\&=&
\log\left(
1 + \frac{\sqrt{2 \rho(\rho^2-z^2)}}{\rho (\rho-z)} \,\delta^{\frac12}+ O(\delta)
\right)\\&=&\frac{\sqrt{2 (\rho^2-z^2)}}{\sqrt\rho\, (\rho-z)} \,\delta^{\frac12}+O(\delta).
\end{eqnarray*}
The proof of the desired result is thereby complete.
\end{proof}

\begin{proof}[Proof of Lemma~\ref{Nmm-POJL}]
We use the explicit representation of the fractional Green function of the ball (see e.g.~\cite[equation~(1.17)]{MR3461641}). Namely, up to a dimensional constant that we neglect, and using the notation in~\eqref{JLSN-12weRo},
\begin{equation*}
\rho_0^{n-2s} G(\rho_0 e,\rho_0 y)=
|e-y|^{2s-n}\int_0^{r_0(\rho_0 e,\rho_0 y)}\frac{t^{s-1}}{(t+1)^{\frac{n}2}}\,dt
.\end{equation*}
Thus, since
$$ \lim_{\rho_0\searrow0}r_0(\rho_0 e,\rho_0 y)=\lim_{\rho_0\searrow0}
\frac{(\rho^2-\rho_0^2)(\rho^2-\rho_0^2)}{\rho^2\rho_0^2|e-y|^2}=+\infty,
$$
we conclude that
$$ \lim_{\rho_0\searrow0}
\rho_0^{n-2s} G(\rho_0 e,\rho_0 y)=
|e-y|^{2s-n}\int_0^{+\infty}\frac{t^{s-1}}{(t+1)^{\frac{n}2}}\,dt,$$
which is the desired result, up to neglecting normalizing constants once again.
\end{proof}
\end{appendix}

\newcommand{\etalchar}[1]{$^{#1}$}


\vfill

\end{document}